\documentclass[preprint]{amsart}
\usepackage{times}
\usepackage{amsfonts}
\usepackage{amsmath}
\usepackage{amscd}
\usepackage{amssymb}
\usepackage{amsxtra}
\usepackage{latexsym}
\usepackage{amsthm}
\usepackage[bookmarks,colorlinks,pagebackref]{hyperref} 
\input{xy}
\xyoption{all}

\theoremstyle{plain}

\swapnumbers
\newtheorem{theorem}[subsection]{Theorem}
\newcommand\Thm[1]{Theorem~\ref{#1}}
\newtheorem{corollary}[subsection]{Corollary}
\newcommand\Cor[1]{Corollary~\ref{#1}}
\newtheorem{lemma}[subsection]{Lemma}
\newcommand\Lem[1]{Lemma~\ref{#1}}
\newtheorem{proposition}[subsection]{Proposition}
\newcommand\Prop[1]{Proposition~\ref{#1}}

\newtheorem*{citedtheorem}{Theorem}

\theoremstyle{definition}

\theoremstyle{remark}

\newtheorem{remark}[subsection]{Remark}
\newcommand\Rem[1]{Remark~\ref{#1}}

\swapnumbers

\newcommand{\emptyprop}{q}

\newcommand \acf{algebraically closed field}

\newcommand \ann[2]{\operatorname{Ann}_{#1}(#2)}

\newcommand \ch{characteristic}
\newcommand \CM{Coh\-en-Mac\-au\-lay}
\newcommand \complet[1]{\widehat {#1}}

\newcommand \ext[4]{\operatorname{Ext}_{#1}^{#2}(#3,#4)}
\renewcommand \hom [3]{\operatorname{Hom}_{#1}(#2,#3)} 
\newcommand \homo{homomorphism}

\renewcommand\iff{if and only if}
\newcommand\into{\hookrightarrow}

\newcommand \inverse[2]{{#1^{-1}(#2)}}
\newcommand \iso{\cong}
\newcommand \loc{{\mathcal {O}}}

\newcommand \map[1]{{\newcommand{\tmpprop}{#1q}  \if\tmpprop\emptyprop \to\else \xrightarrow{{\phantom{i}{#1}\phantom{i}}}\fi}} 
\newcommand \maxim{\mathfrak m}
\newcommand \nat{\mathbb N}

\newcommand \onto{\twoheadrightarrow}
\newcommand \op\operatorname

\newcommand \pr{\mathfrak p}

\let\sub\subseteq

\newcommand \tensor{\otimes}
\newcommand \tor[4]{\operatorname{Tor}^{#1}_{#2}(#3,#4)}

\newcommand \zet{\mathbb Z}

\newcommand \exactseq [5]{0\to{#1}\:\map{#2}\:{#3}\:\map{#4}\:{#5}\to0}
\newcommand \Exactseq [3]{0\to {#1}\to {#2}\to {#3}\to 0}

\newcommand{\commdiagram}[9][]{%
\begin{equation}
{\newcommand{\tmpprop}{#1q} 
\if\tmpprop\emptyprop \relax\else \label{#1}\fi}
\begin{aligned}%
\mbox{
\begin{picture}(130,90)%
\put(120,70){\vector( 0,-1){50}}%
\put(10,80){\vector( 1, 0){100}}%
\put(0,70){\vector( 0,-1){50}}%
\put(10,10){\vector( 1, 0){100}}%
\put(115,80){\makebox(0,0)[l]{$#4$}}%
\put(5,80){\makebox(0,0)[r]{$#2$}}%
\put(115,10){\makebox(0,0)[l]{$#9$}}%
\put(5,10){\makebox(0,0)[r]{$#7$}}%
\put(-3,50){\makebox(0,0)[r]{$#5$}}
\put(123,50){\makebox(0,0)[l]{$#6$}}
\put(60,3){\makebox(0,0)[c]{$#8$}}
\put(60,88){\makebox(0,0)[c]{$#3$}}
\end{picture}}
\end{aligned}
\end{equation}}

\newcommand\commtrianglefront[7][]{%
\begin{equation}
{\newcommand{\tmpprop}{#1q} 
\if\tmpprop\emptyprop \relax\else \label{#1}\fi}
\begin{aligned}%
\mbox{
\begin{picture}(120,80)%
\put(55,68){\vector(-1,-2){30}}
\put(65,68){\vector(1,-2){30}}
\put(30,5){\vector(1,0){60}}
\put(60,75){\makebox(0,0)[c]{$#2$}}
\put(25,5){\makebox(0,0)[r]{$#4$}}
\put(95,5){\makebox(0,0)[l]{$#6$}}
\put(60,0){\makebox(0,0)[c]{$#5$}}
\put(37,43){\makebox(0,0)[r]{$#3$}}
\put(83,43){\makebox(0,0)[l]{$#7$}}
\end{picture}}
\end{aligned}
\end{equation}}

\newcommand\commtriangleback[7][]{%
\begin{equation}
{\newcommand{\tmpprop}{#1q}
\if\tmpprop\emptyprop \relax\else \label{#1}\fi}
\begin{aligned}%
\mbox{
\begin{picture}(120,80)%
\put(55,70){\vector(-1,-2){30}}
\put(65,70){\vector(1,-2){30}}
\put(30,5){\vector(1,0){60}}
\put(60,75){\makebox(0,0)[c]{$#2$}}
\put(25,5){\makebox(0,0)[r]{$#6$}}
\put(95,5){\makebox(0,0)[l]{$#4$}}
\put(60,0){\makebox(0,0)[c]{$#5$}}
\put(37,43){\makebox(0,0)[r]{$#7$}}
\put(83,43){\makebox(0,0)[l]{$#3$}}
\end{picture}}
\end{aligned}
\end{equation}}

\newcommand\binord[1]{\mathfrak o({#1})}
\newcommand\val[2]{\op{val}_{#1}(#2)}
\newcommand\genchi[1]{\chi^{\text{gen}}(\mathcal {#1})}
\newcommand\matlis[1]{#1^\dagger}
\newcommand\assid[2]{\mathfrak {ass}_{#1}(#2)}

\newcommand\aleq[1]{\preceq_{#1}}
\newcommand\order[2]{\op{ord}_{#1}(#2)}

\newcommand\maxel{maximum}
\newcommand\minel{minimum}

\newcommand\ord{\mathbf O}

\newcommand\low[1]{\op{ord}(#1)}

\newcommand\Ssum{\bigoplus}
\newcommand\ssum{\oplus}
\newcommand\bsum{\mathbin{\bar\oplus}}
\newcommand\tsum{\mathbin{\tilde\oplus}}

\newcommand \len[1]{\op{len}(#1)}

\newcommand \genlen[2]{\ell^{\text{gen}}_{#1}(#2)}
\newcommand \lenmod[2]{\op{len}_{#1}(#2)}

\newcommand \lowlen[1]{\op{Lolen}(#1)}
\newcommand \hilen[1]{\op{Hilen}(#1)}

\newcommand  \sch{Grassmanian}
\newcommand  \hdim{height rank}

\newcommand  \hd[1]{\mathfrak l(#1)}

\newcommand  \odmod[2]{\mathfrak o_{#1}(#2)}
\newcommand  \hdmod[2]{\mathfrak l_{#1}(#2)}

\newcommand \gr[1]{\op{Grass}(#1)}
\newcommand \grass[2]{\op{Grass}_{#1}(#2)}
\newcommand \zeroid[1]{\mathbf 0_{#1}}

\newcommand \init[2]{{[#1,#2]}}
  
\newcommand \term[2]{\init{#2}{#1}} 

\title {Semi-addivitity and acyclicity}
\author{Hans Schoutens}
\date\today
\address{Department of Mathematics\\
365 5th Avenue\\
the CUNY Graduate Center\\
New York, NY 10016, USA}
\email{hschoutens@citytech.cuny.edu}

\begin{document}
\begin{abstract} We generalize the notion of length to an ordinal-valued invariant defined on the class of finitely generated modules over a Noetherian ring. A key property of this invariant is its semi-additivity on short exact sequences. As an application, we prove some general acyclicity theorems.
\end{abstract}

\maketitle



\section{Introduction}
The \emph{length}  $\len M$ of an Artinian, finitely generated module $M$ is      defined as the longest chain of submodules in $M$. Since we have the descending chain condition, such a chain is finite, and hence can be viewed as a finite ordinal. Hence we can immediately generalize this by transfinite induction to arbitrary Artinian modules, getting an ordinal-valued length function. To remain in the more familiar category of finitely generated modules, observe that at least over a complete Noetherian local ring, the latter is anti-equivalent with the class of Artinian modules via Matlis duality. We could have used this perspective (which we will discuss  in \cite{SchOrdLen}), but a moment's reflection directs us to a simpler solution: just reverse the order. Indeed, if we view the class of all submodules\footnote{Throughout the paper, $R$ is a finite dimensional Noetherian ring and $M,N,\dots$ are  finitely generated $R$-modules.} of $M$, the \emph\sch\ $\grass RM$, as a partially ordered set by reverse inclusion, then $\grass RM$ admits the descending chain condition, and hence any subchain is well-ordered, that is to say an ordinal. We then simply define $\len M$ as the supremum of all such chains/ordinals in $\grass RM$. Viewed as a module over itself, this yields the \emph{length} $\len R$ of a Noetherian ring $R$.

The key property of ordinary length is its additivity on short exact sequences. An example like $\exactseq\zet  2\zet{}{\zet/2\zet}$ immediately shows this can no longer hold in the general transfinite case. Moreover, even the formulation of additivity becomes problematic since ordinal sum is no longer commutative. There does exist a different, commutative sum, called in this paper the \emph{shuffle sum} $\ssum$ (see Appendix~\ref{s:ssum}), which, as our  main result shows, also plays a role:

\begin{citedtheorem}[Semi-additivity, \Thm{T:semadd}]
If $\Exactseq MNQ$ is exact, then $\len Q+\len N\leq \len M\leq \len Q\ssum\len N$.
\end{citedtheorem}

To appreciate the power of this result, notice that we instantaneously recover Vasconcelos' observation that a surjective endomorphism is also injective (see \Cor{C:Vasc} below). Extending this observation, we formulate in \Thm{T:acycl} a general  acyclicity criterion in terms of a certain ``ordinal Euler \ch'' (since we do not have a well-behaved subtraction, this is in fact a pair of two ordinals). 
As an application, we see that $f\colon M\to N$ is injective  whenever $\len{N}=\len{M}\ssum\len{\op{coker}f}$. We also get a new proof and a generalization of a result by Miyata in \cite{Miy}:  any exact sequence of the form  $M\to M\oplus N\to N\to 0$ must be split exact.
The length of a module encodes quite some information of the module. For instance, its degree, that is to say, the highest power $\omega^d$ occurring in $\len M$, is precisely the dimension of $M$ (\Thm{T:dim}). It follows that a Noetherian ring has length $\omega^d$ \iff\ it is a $d$-dimensional domain.  In \cite{SchOrdLen}, we give a formula  expressing  length via (zero-th) local cohomology, but in this paper, we content ourself with proving the following special case: all associated primes of $M$ have the same dimension---that is to say, $M$ is \emph{unmixed}---\iff\ $\len M$ is a monomial of the form $a\omega^d$; moreover, $a$ is then the generic length $\genlen{}M$ and $d$ its dimension (\Thm{T:unm}).
%
Since ordinals and ordered structures are not the usual protagonists in commutative algebra, the paper starts with a section on this, and in an appendix, I explain   shuffle sums. In \S\ref{s:mod}, I prove semi-additivity, and the the next sections contain some applications.

\section{Notation and generalities on ordered sets}
  An \emph{ordered}
set $P$ (also called a \emph{partially ordered} set or \emph{poset}), is a set
together with a reflexive, antisymmetric and transitive binary relation
$\leq_P$, called the \emph{ordering} of $P$, and almost always written as 
$\leq$, without a subscript. A partial order is \emph{total} if for any two elements $a,b\in P$ either $a\leq b$ or $b\leq a$. A subset $C\sub P$ is called a \emph{chain}, if its induced order is total.
 If $a\leq b$, then
we may express
this by saying that \emph{$a$ is below $b$};
if $a<b$ (meaning that $a\leq b$ and $a\neq b$), we also say that \emph{$a$ is
strictly below $b$}.
More generally, for subsets
$A,B\sub P$, we say $A$ is \emph{below} $B$, 
and
write $A\leq B$, to mean that $a\leq b$ for
all $a\in A$ and all $b\in B$.

The   \emph{initial closed interval}  determined by $a\in P$ is by definition the set of 
$b\in  P$ with $b\leq a$ and will be denoted   $\init Pa$. Dually, the
 \emph{terminal closed interval} of $a$, denoted  $\term Pa$,
is the collection of all $b\in P$
with  $a\leq b$.

\subsection{Ordinals}\label{s:Ord}
A partial ordering is called a \emph{partial well-order} if it has the
descending chain condition, that is to say, any descending chain must
eventually be constant. A total order is a well-order  if every non-empty
subset has a minimal element.  

Recall that an \emph{ordinal} is an equivalence class, up to an order-preserving  isomorphism, of a total well-order.  The class of all ordinals is denoted $\ord$;  any bounded subset of $\ord$ has  then  an infimum and a supremum. 
For generalities on ordinals, see any elementary textbook on set-theory (but see also 
Appendix~\ref{s:ssum}   for the notion of shuffle sum). Let me remind the reader of the fact that ordinal sum (see   \S\ref{s:sum} below) is not commutative: $1+\omega\neq\omega+1$ since the former is just $\omega$. We
will adopt the usual notations except for one, where we will reverse the
order. Frankly, the common notation
 for multiplication goes against any (modern) sense of logic, aesthetics or
analogy. Therefore, in this paper, $\alpha\beta$ will simply mean \emph{$\alpha$
copies of $\beta$}, that is to say, $\alpha\beta$ is equal to the lexicographic
ordering on $\alpha\times\beta$. After all, $2\omega$ should mean `two omega', that is to say,   $\omega+\omega$.

All ordinals considered will be less than $\omega^\omega$, and hence can be written uniquely in Cantor normal form  
\begin{equation}\label{eq:Cantor}
\alpha=\sum_{i=0}^da_i\omega^i
\end{equation} 
 with $a_i\in\nat$. We call the least $i$ (respectively, the largest $i$) such that $a_i\neq 0$ the \emph{order} $\order{}\alpha$ (respectively, the \emph{degree} $\deg\alpha$) of $\alpha$; the sum of all $a_i$ is called its \emph{valence} $\val{}\alpha$.
An ordinal $\alpha$ is called a \emph{successor ordinal}  if it has an immediate predecessor, denoted simply $\alpha-1$. This is equivalent with $\order{}\alpha=0$. 
Given any $e\geq 0$, we will write 
\begin{equation}\label{eq:splitord}
\alpha^+_e:=\sum_{i=e}^d a_i\omega\qquad\text{and}\qquad \alpha_e^-:=\sum_{i=0}^{e-1} a_i\omega^i,
\end{equation} 
where the $a_i$ are given by \eqref{eq:Cantor}. In particular, $\alpha=\alpha^+_e+\alpha_e^-$. We will write $\alpha\aleq e\beta$ (respectively, $\alpha=_e\beta$) if $\alpha^+_e\leq\beta_e^+$ (respectively, $\alpha^+_e=\beta_e^+$). To  make   statements uniform in all $e$, we assign degree $-1$ to $0$, and dimension $-1$ to a zero module. We will need:

\begin{lemma}\label{L:eineq}
Let  $\alpha$ and $\beta$  be ordinals and let $e\in\nat$ be such that $\order{}\alpha\geq e$ and $\alpha=_e\beta$.  If for some ordinal $\lambda$ we have $\alpha+\lambda\leq  \beta$, then $\lambda$ has degree at most $e-1$.
\end{lemma}
\begin{proof}
We use the notation from \eqref{eq:splitord}. By assumption, $\alpha=\alpha^+_e=\beta^+_e$. 
Let $d$ be   larger than the degree of any of the ordinals involved.  For $e\leq q\leq d$, we prove by downward induction on $q$ that $\deg\lambda\leq q$. The case $q=d$ is clear, so assume $\deg\lambda\leq q$. Write $\alpha=\sum_{i\geq e} a_i\omega^i$ and $\lambda=\sum_{i=0}^ql_i\omega^i$. The coefficient of $\omega^q$ in $\alpha+\lambda$ is equal to $a_q+l_q$, and since $\alpha^+_q=\beta^+_q$, this is at most $a_q$, proving that $l_q=0$. 
\end{proof}
%


\subsection{The length of a partial well-order}\label{s:odim} 

Let $P$ be a partial well-order. We define 
    the \emph{\hdim} $\hdmod P \cdot$ on $P$ by transfinite induction as follows: at successor stages, we say that $\hdmod P a\geq\alpha+1$, if there exists $b\leq a$ with $\hdmod P b\geq\alpha$, and at limit stages, that $\hdmod Pa\geq\lambda$, if there exists for each $\alpha<\lambda$ some $b_\alpha\leq a$ with $\hdmod P{b_\alpha}\geq\alpha$. We then say that $\hdmod Pa=\alpha$ if $\hdmod Pa\geq\alpha$ but not $\hdmod Pa\geq \alpha+1$. In particular, $\hdmod Pa=0$ \iff\ $a$ is a minimal element of $P$. For a subset $A\sub P$, we set $\hdmod PA$ equal to the supremum of all $\hdmod Pa$ with $a\in A$. Finally, we define the \emph{(ordinal) length} of $P$ as $\len P:=\hdmod PP$. In  particular, if $P$ has a \maxel\ $\top$, then $\len P=\hdmod P\top$.

\begin{lemma}\label{L:AB}
Let $P$ be a  partial well-order and let 
$A,B\sub P$ be subsets. If $A\leq B$, then 
\begin{equation}
\label{eq:AB}
\hdmod AA+\hdmod BB \leq \hdmod PB
\end{equation}
\end{lemma}
\begin{proof}
Let $\alpha:=\hdmod AA=\len A$. Since  $\hdmod PB$ is the supremum of all $\hdmod Pb$ with $b\in B$, it suffices to show   that
\begin{equation}\label{eq:Ab}
\alpha+\hdmod Bb \leq  \hdmod Pb.
\end{equation}
We will prove \eqref{eq:Ab} by induction on $\beta:=\hdmod Bb$. Assume
first that $\beta=0$. Let $\theta:=\hdmod PA$. Since $\alpha$
is the supremum of all $\hdmod Aa$ for $a\in
A$, and since $\hdmod Aa\leq\hdmod Pa$, we get $\alpha\leq\theta$. Since $A\leq b$, we have
$\theta\leq\hdmod Pb$, and hence we are done in this case.

Next,  assume   $\beta$ is a successor ordinal, and denote its predecessor by $\beta-1$. By definition, there exists $b'\in B$ below $b$ such that $\hdmod B{b'}\geq\beta-1$.
By induction,
we get $\hdmod  P{b'}\geq \alpha+\beta-1$. This in turn shows that
$\hdmod Pb$ is at least $\alpha+\beta$. Finally, assume $\beta$ is a limit
ordinal. Hence for each $\gamma<\beta$, there
exists
$b_\gamma\in B$ below $b$ such that $\hdmod B{b_\gamma}=\gamma$. By induction,
$\hdmod P{b_\gamma}\geq\alpha+\gamma\leq \alpha+\beta$ and hence also  $\hdmod Pb\geq \alpha+\beta$.
\end{proof}

\subsection{Sum Orders}\label{s:sum}
By the  \emph{sum} $P+Q$ of two  partially ordered sets $P$ and $Q$, we mean the partial order induced on their disjoint union $P\sqcup Q$ by declaring any element in $P$ to lie below any element in $Q$. In fact, if $\alpha$ and $\beta$ are ordinals, then their sum is just $\alpha\sqcup\beta$, customarily denoted $\alpha+\beta$.
We may represent elements
in the disjoint union $P\sqcup Q$ as pairs $(i,a)$ with $i=0$ if $a\in P$
and $i=1$ if $a\in Q$. The
ordering  $P+ Q$ is  then the lexicographical
ordering on such pairs, that is to say, $(i,a)\leq (j,b)$ if $i<j$ or if $i=j$
and $a\leq b$. 

\begin{proposition}\label{P:sum}
If $P$ and $Q$ are partial well-orders, then so is $P+Q$.
If $P$ has moreover a \maxel, then 
$\len{P+Q}=\len P+\len Q$.
\end{proposition}
\begin{proof}
We leave it as an exercise to show that $P+ Q$ is a partial well-order.  
Let $\pi:=\hdmod P{\top_P}= \len P$. For a pair $(i,a)$ in  $P+ Q$, let $\nu(i,a)$ be
equal to $\hdmod  Pa$ if $i=0$
and to $\pi+\hdmod Qa$ if $i=1$. The assertion will follow once we showed that
$\nu(i,a)=\hd{i,a}$, for all $(i,a)\in P+Q$, where we wrote $\hd{i,a}$ for $\hdmod {P+Q}{i,a}$. We use transfinite induction. If $i=0$, that is
to say, if $a\in P$, then the claim is easy to check, since no element from $Q$
lies below $a$. So we may assume $i=1$ and $a\in Q$.  Let $\alpha:=\hdmod  Qa$
and suppose first that $\alpha=0$. Since any element of $P$ lies below $a$, in
any case $\pi\leq\hd{i,a}$. If this were strict, then there would be
an element  $(j,b)$ below $(i,a)$ of \hdim\ $\pi$.
Lest
$\hd{0,\top_P}$ would be bigger than $\pi$, we must have $j=1$ whence $b\in Q$. Since $b\leq_Q a$, we get  $\hdmod Q a\geq1$, contradiction.
  This
concludes the case $\alpha=0$, so  assume $\alpha>0$. We leave the limit case to
the reader and assume moreover that $\alpha$ is a successor ordinal. Hence there
exists some  $b\in Q$   below $a$    with $\odmod Q{b}=\alpha-1$. By
induction, $\hd{1,b}=\nu(1,b)=\pi+\alpha-1$, and hence 
$\hd{1,a}\geq\pi+\alpha$. By a similar argument as above, one then easily
shows that this must in fact be an equality, as we wanted to show.
\end{proof}

\subsection{Product Orders}\label{s:prod}
The \emph{product} of two  partially ordered sets $P$ and $Q$ is defined to be
the Cartesian product $P\times Q$ ordered by the rule $(a,b)\leq (a',b')$
\iff\ $a\leq a'$ and $b\leq b'$. Note that this is not the lexicographical
ordering, even if  $P$ and $Q$ are total orders. The map $(a,b)\mapsto (b,a)$
is an order-preserving bijection between $P\times Q$ and $Q\times P$. It is
easy to check that  if both $P$ and $Q$ are partial well-orders, then so is $P\times
Q$. 

For the next result, we make use of the ordinal sum $\ssum$ defined in
Appendix~\ref{s:ssum} below. 
%

\begin{theorem}[Product Formula]\label{T:prod}
Given partial well-orders $P$ and $Q$, we have an equality $\len{P\times Q}= \len P\ssum\len Q$.
\end{theorem}
\begin{proof}
We prove the more general fact that   
\begin{equation}\label{eq:odimprod}
\hd{a,b}=\hd a\ssum\hd b 
\end{equation}
 for all
$a\in P$ and $b\in Q$, from which the assertion follows by taking suprema over
all elements in $P$ and $Q$. Note that we have not written superscripts   to denote on which ordered set
the \hdim\ is calculated since this is clear
from the context. To prove \eqref{eq:odimprod}, we may assume by
transfinite induction that it holds for all pairs $(a',b')$ strictly below $(a,b)$.
Put $\alpha:=\hd a$, $\beta:=\hd b$ and $\gamma:=\hd{a,b}$. Since
$\hd{a,b}=\hd{b,a}$, via the isomorphism $P\times Q\iso Q\times P$, we may
assume that $\low\alpha\leq\low\beta$ whenever this assumption is required
(namely, when dealing with limit ordinals). Here $\low\alpha$ denotes the \emph{order} of $\alpha$, that is to say, the 
lowest exponent in the Cantor normal form; see
\S\ref{s:ssum} for details. 

We
start with proving the
inequality $\alpha\ssum\beta\leq\gamma$. If   $\alpha=\beta=0$ then  
$a$ and $b$ are minimal elements in respectively $P$ and $Q$, whence so is $(a,b)$ in $P\times Q$, that is to say,  $\gamma=0$. So we may
assume,
after perhaps exchanging $P$ with $Q$ that $\alpha>0$. Suppose $\alpha$ is a successor
ordinal. Hence there exists $a'<a$  in $P$ 
with $\hd{a'}=\alpha-1$.
By induction, $\hd{a',b}=(\alpha-1)\ssum\beta$ and hence 
$\gamma=\hd{a,b}$ is at least $((\alpha-1)\ssum\beta)+1=\alpha\ssum\beta$,
where the last equality follows from Theorem~\ref{T:ssum}.
If $\alpha$
is a limit ordinal, then there exists for each $\delta<\alpha$ an element
$a_\delta<a$ of \hdim\ $\delta$. By
induction
$\hd{a_\delta,b}=\delta\ssum\beta$ and hence $\hd{a,b}$ is at least
$\alpha\ssum \beta$ by  Theorem~\ref{T:ssum}.  This concludes the
proof that $\alpha\ssum\beta\leq \gamma$. 

For the converse inequality, assume first that $\gamma$ is a
successor ordinal. By definition, there exists $(a',b')<(a,b)$ in $P\times Q$
of \hdim\   $\gamma-1$.
By induction, $\gamma-1=\hd{a',b'}=\hd {a'}\ssum\hd{b'}$, from which we get $\gamma\leq \alpha\ssum\beta$. A similar argument can be
used to treat the limit case and the details are left to the reader.
\end{proof}

%
%

\subsection*{Increasing functions}
We conclude this section with the behavior of \hdim\   under an increasing
function. Let $f\colon P\to Q$ be an increasing (=order-preserving) map
between ordered sets. We say that $f$ is
\emph{strictly
increasing}, if $a<b$ then $f(a)<f(b)$. For instance, an increasing, injective  map  is strictly increasing.
%

\begin{theorem}\label{T:inc}
Let $f\colon P\to Q$ be a strictly increasing map between partial well-orders. If $P$ has  a \minel\  $\bot_P$, then  
\begin{equation*}
\hdmod Q{f(\bot_P)}+\hdmod Pa\leq \hdmod Q{f(a)}.
\end{equation*}
for all    $a\in P$.
\end{theorem}
\begin{proof}
From the context, it will  be clear in which ordered set we calculate
the rank and hence we will drop the superscripts.   Let $\gamma:=\hd{f(\bot)}$.  We
   induct  on $\alpha:=\hd a$, where the case
$\alpha=0$ holds trivially.    We  leave the limit case to the reader and  
assume that $\alpha$
is a successor ordinal. By definition, there exists $b<a$
with $\hd{b}=\alpha-1$. By induction, the \hdim\ of
$f(b)$ is at least $\gamma+\alpha-1$. By assumption,  $f(b)<f(a)$, showing that  $f(a)$ has \hdim\ at least $\gamma+\alpha$. 
\end{proof}

Even in the absence of a \minel, the inequality   still holds,  upon replacing the first ordinal in the formula by the minimum of the ranks
of all $f(a)$ for $a\in P$. In particular, \hdim\  
always
increases.

\section{Semi-additivity}\label{s:mod}
 Let $R$ be a ring and $M$ a Noetherian $R$-module.
The \emph\sch\ of $M$ (over $R$) is by definition the collection $\grass  RM$ of all submodules
of $M$, ordered  by reverse inclusion.  Note that this is a natural generalization
of the notion of a \sch\ of a vector space $V$ over a field $K$. The \hdim\ of $\gr M$ will be called the \emph{length} $\lenmod RM$ of $M$ as
an $R$-module. This is well-defined, since   $\grass RM$ is   a well-partial order.
Thus, for   $N\sub M$, we have $\hd N\geq \alpha+1$, if there exists $N'$   containing $N$ with $\hd{N'}\geq\alpha$. 
 Since the initial closed interval 
$\init {\gr M}N$   is isomorphic to $\grass R {M/N}$, we get
 \begin{equation}\label{eq:quot}
 \hd N=\hdmod {\grass R M}N=\hdmod{\grass R{M/N}}{\zeroid {M/N}}=\lenmod R {M/N}
 \end{equation}
and hence in particular $\hd{\zeroid M}=\lenmod R M$, where $\zeroid M$ denotes the
zero module of $M$. Similarly, $\term N{\grass R M}$ consists of all submodules of $M$ contained in $N$, whence is equal to $\grass R N$. Note that if $I$ is an ideal in the annihilator of $M$, then $\grass RM=\grass {R/I}M$, so that in order to calculate the length or the order dimension of $M$, it makes no difference whether we view it as an $R$-module or as an $R/I$-module. We call the \emph{length} of $R$, denoted $\len R$, its length when viewed as a module over itself. 
 Hence, the length  of $R/I$ as an $R$-module is the same as that of $R/I$ viewed as a ring. We define the \emph{order}, $\order RM$, and \emph{valence}, $\val RM$, as the respective order and valence of $\lenmod RM$.

\begin{theorem}[Semi-additivity]\label{T:semadd}
If $\Exactseq NMQ$ is an exact sequence of Noetherian $R$-modules, then 
\begin{equation}\label{eq:lensemadd}
  \lenmod R Q+\lenmod R N\leq \lenmod R M\leq \lenmod R Q\ssum \lenmod R N
\end{equation}
Moreover, if the sequence is split, then the last inequality is an equality.
\end{theorem}
\begin{proof}
The last assertion   follows from the first, \Thm{T:prod}, and    the fact that then 
$$ \grass R N\times\grass R Q\sub \grass R M.$$
 To prove the lower estimate, let $A$ be the initial closed interval $\init {\grass R M}N$ and let $B$ be the terminal closed interval  $\term {\grass R M}N$. By our discussion above, $A=\grass R{M/N}=\gr Q$, since $M/N\iso Q$,  with   \maxel, viewed in $\grass R Q$, equal to   $\zeroid Q$. By the same discussion,  $B=\grass R N$ with \maxel\ $\zeroid N$. Since $A\leq B$, we may apply Lemma~\ref{L:AB}   to get an inequality 
$$
\hdmod {\grass R Q}{\zeroid Q} + \hdmod {\grass R N}{\zeroid N}\leq \hdmod{\grass R M}{\zeroid M},
$$
from which the assertion follows.

To prove the upper bound, let $f\colon \grass RM\to \grass RN\times\grass RQ$ be the map sending a submodule $H\sub M$ to the pair $(H\cap N,\pi(H))$, where $\pi$ denotes the morphism $M\to Q$. It is not hard to see that this is an increasing function. Although it is in general not injective, I claim that $f$ is strictly increasing, so that we can apply  \Thm{T:inc}. Together with \Thm{T:prod}, this gives us the desired inequality. So remains to verify the claim: suppose $H<H'$ but $f(H)=f(H')$. Hence $H'\varsubsetneq H$, but $H\cap N=H'\cap N$ and $\pi(H)=\pi(H')$. Applying the last equality to an element $h\in H\setminus H'$, we get $\pi(h)\in \pi(H')$, whence $h\in H'+N$. Hence, there exists $h'\in H'$ such that $h-h'$ lies in $H\cap N$ whence in $H'\cap N$. This in turn would mean $h\in H'$, contradicting our assumption on $h$.
\end{proof}

Using that $\alpha+n=\alpha\ssum n$, when $n$ is finite, we immediately get:

\begin{corollary}\label{C:finlenadd}
If $N$ is a submodule of $M$ of finite length, then 
$$\lenmod RM=\lenmod R{M/N}+  \lenmod RN.\qed$$
\end{corollary}

%
%

%
%
%
\begin{corollary}\label{C:regid}
Let $R$ be a Noetherian ring. If $x$ is an $R$-regular element and $I\sub R$ an
arbitrary ideal, then
 $$
\len{R/xR}+\len{R/I}\leq \len{R/xI}.
$$
\end{corollary}
\begin{proof}
Apply Theorem~\ref{T:semadd} to the exact sequence
$$
\exactseq {R/I} x{R/xI} {} {R/xR}.
$$
\end{proof}
 
Applying Corollary~\ref{C:regid} to the zero ideal and observing that $\alpha+\beta=\beta$ \iff\ $\deg\alpha<\deg\beta$, we get:

\begin{corollary}\label{C:hyp}
If $x$ is an $R$-regular element, then the  degree of $\len{R/xR}$ is
 strictly less than the degree of $\len R$.\qed
\end{corollary}

\begin{theorem}\label{T:dim}
Let $M$ be a   finitely generated module over a Noetherian ring $R$. Then the degree of $\lenmod RM$ is equal to the dimension of $M$. In particular,   $R$ is a $d$-dimensional domain \iff\ $\len R=\omega^d$.
\end{theorem}
\begin{proof}
Let $\mu$   be the   length of $M$ and $d$ its dimension. 
 We start with proving the inequality  
\begin{equation}\label{eq:dimlow}
 \omega^d\leq \mu. 
\end{equation}
 We will do this first for $M=R$, by induction on  $d$, where the case $d=1$ is clear, since $R$ does not have finite length. Hence we may assume $d>1$.
Taking the residue modulo a $d$-dimensional prime ideal (which only can
lower   length), we may assume that $R$ is a domain. Let $\pr$ be a $(d-1)$-dimensional prime ideal  and let $x$ be a non-zero element in $\pr$. 
By Corollary~\ref{C:hyp},
 the degree of   $\len{R/xR}$ is at most $\deg\mu-1$.  
By induction,
 $\omega^{d-1}\leq\len{R/\pr} \leq \len{R/xR}$, whence $d-1\leq \deg\mu-1$, proving \eqref{eq:dimlow}.

For $M$ an arbitrary $R$-module, let $\pr$ be a $d$-dimensional associated prime of $M$, so that we can find an exact sequence
\begin{equation}\label{eq:asspr}
\Exactseq {R/\pr}M{\bar M}.
\end{equation}
With $\bar\mu:=\len{\bar M}$, \Thm{T:semadd} yields
\begin{equation}\label{eq:assprlen}
\bar \mu +\len{R/\pr}\leq \mu\leq \bar \mu\ssum\len{R/\pr}.
\end{equation}
In particular, $\omega^d\leq\len{R/\pr}\leq \mu$, proving \eqref{eq:dimlow}. 

Next we show that 
\begin{equation}\label{eq:dimhi}
\mu<\omega^{d+1},
\end{equation}
 again by induction on $d$. 
Assume   first that $M=R$ is a domain. Since   $R/I$ has then dimension at most $d-1$ for any non-zero ideal $I$, we get $\len{R/I}<\omega^d$ by our induction hypothesis. By \eqref{eq:quot}, this means that any non-zero ideal has \hdim\ less than $\omega^d$, and hence $R$ itself has length at most $\omega^d$. Together with \eqref{eq:dimlow}, this already proves one direction in the second assertion.
For the general case, we do a second induction, this time on $\mu$. With $\pr$ as above, a $d$-dimensional associated prime of $M$, we get $\mu\leq \bar \mu\ssum\len{R/\pr}$ by \eqref{eq:assprlen}.  By what we just proved, $\len{R/\pr}=\omega^d$, and hence by induction $\mu\leq\bar \mu\ssum\omega^d<\omega^{d+1}$. The first assertion is now immediate from \eqref{eq:dimlow} and \eqref{eq:dimhi}.

Conversely, if $R$ has length $\omega^d$, then for any non-zero ideal $I$, the length of $R/I$ is strictly less than $\omega^d$, whence its dimension is strictly less than $d$ by what we just proved. This shows that $R$ must be a domain.
\end{proof}

%

Let $\genlen RM$ be the \emph{generic length} of a Noetherian module $M$, defined as the sum
$$
\genlen RM=\sum_{\op{dim}\pr=\op{dim}(M)}\len{M_\pr}.  
$$
Note that we only have a non-zero contribution in this sum if $\pr$ is a minimal prime of $M$, and the corresponding localization $M_\pr$ then has finite length, so that $\genlen{}M$ is well-defined. If $M$ has finite length, $\genlen{}M=\len M$. With the notation from \eqref{eq:splitord}, we get

\begin{proposition}\label{P:genlen}
For $d=\op{dim}(M)$, we have $\lenmod RM^+_d=\genlen RM\omega^d$.
\end{proposition} 
\begin{proof}
We induct on $\mu:=\len M$, where the case for finite $\mu$ is clear. By \Thm{T:dim} we can write $\mu=a\omega^d+\mu^-$. Let $\pr$ be a $d$-dimensional associated prime of $M$, so that we have an exact sequence~\eqref {eq:asspr} 
and let $\bar \mu:=\len{\bar M}$. By induction, we have $\bar\mu=\genlen{}{\bar M}\omega^d+\bar\mu^-$. Applying \Thm{T:semadd} to \eqref{eq:asspr} yields inequalities $\bar\mu+\omega^d\leq\mu\leq \bar\mu\ssum\omega^d$. Looking at the coefficients of the degree $d$ terms,  this implies $a=\genlen{}{\bar M}+1$. On the other hand, localizing \eqref{eq:asspr} at $\pr$ and taking lengths gives $\len{M_\pr}=\len{\bar M_\pr}+1$, whereas localizing at any other $d$-dimensional prime ideal $\mathfrak q$   gives $M_{\mathfrak q}=\bar M_{\mathfrak q}$, showing that $\genlen{}M=\genlen {}{\bar M}+1=a$.
\end{proof}

\begin{remark}\label{R:dim}
In \cite{SchOrdLen}, we will extend this formula by calculating all coefficients in $\lenmod RM$. As a corollary, we will obtain that the order of $M$ is the minimal dimension of an associated prime of $M$. In this paper, we only prove the following  consequence of this characterization, where we call a module $M$ \emph{unmixed}, if all its associated primes have the same dimension (equivalently, if  all non-zero modules have the same dimension).
\end{remark}

\begin{theorem}\label{T:unm}
A module is unmixed \iff\ its length is a monomial, that is to say, of the form $a\omega^d$. Moreover, $a$ is then the generic length of the module and $d$ its dimension.
\end{theorem} 
\begin{proof}
The second assertion is just \Prop{P:genlen}. 
Let $d$ be the dimension of $M$. Suppose $M$ is unmixed and let $\mu$ be its length. By \Thm{T:dim}, we can write $\mu=a\omega^d+\mu^-$, with $a$ a positive integer and  $\mu^-:=\mu^-_{d-1}$ as in \eqref{eq:splitord}. We need to show that $\mu^-=0$. Suppose it is not, so that there must exist a non-zero submodule $N$ with $\hd N=a\omega^d$. By \eqref{eq:quot}, we have $\len{M/N}=a\omega^d$. By \Thm{T:semadd}, with $\nu:=\len N$, we get $a\omega^d+\nu\leq\mu=a\omega^d+\mu^-$. Hence $\nu\leq\mu^-$, and so $\nu$   has degree at most $d-1$, which means that $\op{dim}(N)<d$ by \Thm{T:dim}. Since $M$ is unmixed, we must have $N=0$, contradiction.

Conversely, suppose $\mu=a\omega^d$. Suppose there exists a non-zero sub-module  $N\sub M$  of dimension at most $d-1$. Let $\nu$ and $\gamma$ be the respective lengths of $N$ and $M/N$. By semi-additivity, we have 
$$\gamma+\nu\leq a\omega^d\leq \gamma\ssum\nu.$$
By \Thm{T:dim}, the degree of $\nu$ is at most $d-1$. Hence for the first equality to hold, we must have $\gamma<a\omega^d$. However,   $\gamma\ssum\nu$ is then also strictly less than $a\omega^d$, contradicting the second inequality.
\end{proof} 

Immediately from \Thm{T:dim} and \Thm{T:unm}, since $\omega^d+\nu=\omega^d\ssum \nu$, for any ordinal $\nu$ of degree at most $d$, we get

 \begin{corollary}\label{C:unm}
If $\Exactseq NMQ$ is exact and $Q$ is unmixed with $\op{dim}Q=\op{dim}M$, then $\len M=\len Q\ssum \len N$.\qed
\end{corollary} 

There is another measure for how long   a partial well-order $P$ is, namely the maximal length of a chain in $P$. More precisely, given a chain  $\mathcal C$ in $P$,  let $\binord {\mathcal C}$ be the ordinal giving the order type of ${\mathcal C}$. Assume $P$ has a \maxel\ $\top$. Define the \emph{chain length} of $P$ as the supremum of all $\binord {\mathcal C}$, where ${\mathcal C}$ runs over all chains in $P$ not containing $\top$. We have to omit  $\top$  here to conform with the notion of   length for   finite chains as one less than their cardinality. The chain length can be smaller than the length: for each $n$, let $C_n$ be a chain of length $n$,   let $P$ be obtained from the disjoint union of all $C_n$ by adding two more elements $a<b$ above all these. In particular, as each chain in $P$ is finite, but of arbitrarily large size, its chain length is $\omega$. However, the \hdim\ of $a$ is $\omega$, and hence $\len P=\hdmod Pb=\hdmod Pa+1 =\omega+1$. This phenomenon cannot occur in \sch{s}, and, in fact,   the chain length is even a maximum.

\begin{theorem}\label{T:chain}
The length of $M$ is equal to the chain length of $\grass RM$. In fact, $\lenmod RM$ is the maximum of all $\binord {\mathcal C}$, where ${\mathcal C}$ runs over all chains of non-zero submodules in $\grass RM$. 
\end{theorem} 
\begin{proof}
Let ${\mathcal C}$ be a chain of non-zero submodules and let $\rho:=\binord {\mathcal C}$. Let $M_\alpha$ for $\alpha<\rho$ be    the $\alpha$-th element in this chain. In particular, $\zeroid M$ is the $\rho$-th element in the chain ${\mathcal C}\cup\{\zeroid M\}$.  By definition, $\hd {M_\alpha}\geq \alpha$. Hence $\len M=\hd{\zeroid M}\geq \binord M$.
So the result will follows once we prove the existence of a chain of non-zero modules of length $\mu:=\len M$, and we do this  by induction on $\mu$. For $\mu<\omega$, this is just the classical Jordan-Holder theorem for modules of finite length. If $\mu$ has valence at least two, we can write it as $\bar\mu\ssum\omega^e$, where $e$ is the order of $\mu$. Let $N$ be a submodule of \hdim\ $\bar\mu$, so that $M/N$ has length $\bar\mu$ by \eqref{eq:quot}. If $\nu:=\len N$, then by semi-additivity, we get $\bar\mu+\nu\leq \bar\mu\ssum\omega^e\leq\bar\mu\ssum\nu$. The latter implies that $\omega^e\leq\nu$, and by the former inequality, it cannot be bigger either. By induction we can find a chain of non-zero submodules in $\grass{}{M/N}$ of length $\bar\mu$, and hence, upon lifting these to $M$, we get a chain in $\grass{}M$ of submodules strictly containing $N$. Induction also gives a chain of non-zero submodules of $N$ of length $\omega^e$, and putting these two chains together, we get a chain of length $\mu$.  

So remains the case that $\mu=\omega^d$.  We prove this case independently by induction on $d$, where the case $d=1$ is classical: if there were no infinite chains,  then $M$ is both Artinian and Noetherian, whence of finite length (\cite[Proposition 6.8]{AtiMac}). Put $M_0:=M$ and choose a submodule $M_1$ of \hdim\ $\omega^{d-1}$. Applying the induction hypothesis to  $M/M_1$, we can find, as above, a chain ${\mathcal C}_1$ of submodules strictly containing $M_1$, of length $\omega^{d-1}$. Since $M_1$ is non-zero, it has dimension $d$, whence its length must be at least $\omega^d$ by \Thm{T:dim}, and therefore, by semi-additivity, equal to it.   Choose a submodule $M_2$ of $M_1$ of \hdim\ $\omega^{d-1}$ in $\grass{}{M_1}$, and as before, find a $\omega^{d-1}$-chain ${\mathcal C}_2$ in this \sch\    of submodules strictly containing $M_2$. Continuing in this manner, we get a descending chain $M_0\varsupsetneq M_1\varsupsetneq M_2\varsupsetneq\dots$ and $\omega^{d-1}$-chains ${\mathcal C}_n$ from $M_n$ down to $M_{n+1}$. The union of all these chains is therefore a chain of length $\omega^d$, as we needed to construct.
\end{proof}

\section{Acyclicity}
For the remainder of this paper $R$ is a Noetherian ring, $d$   its dimension, and $\rho$ its length. Furthermore,   $M$, $N$, \dots are finitely generated modules over $R$, of length $\mu$, $\nu$, etc.
We start with reproving the observation of Vasconcelos \cite{VasFlat} that a surjective endomorphism on a Noetherian module must be an isomorphism (the usual proof uses the determinant trick; see for instance \cite[Theorem 2.4]{Mats}).

\begin{corollary}\label{C:Vasc}
Any surjective endomorphism  is an isomorphism.
\end{corollary}
\begin{proof}
Let $M\to M$ be a surjective endomorphism with kernel  $N$, so that we have an exact sequence $\Exactseq NMM$, and therefore, by \Thm{T:semadd}, an inequality $\len M+\len N\leq \len M$. By simple ordinal arithmetic, this implies $\len N=0$, whence $N=0$.
\end{proof}

\begin{remark}\label{R:Vasc}
Our argument in fact proves that any surjection between modules of the same length must be an isomorphism, or more generally, if $f\colon M\to N$ is an epimorphism and $\len  M\leq \len N$, then $f$ is an isomorphism and $\len M=\len N$.
\end{remark}

\begin{corollary}\label{C:subim}
If $N$ is a homomorphic image of $M$ which contains a submodule isomorphic to $M$, then $M\iso N$.
\end{corollary} 
\begin{proof}
  Since $M\into N$, semi-additivity yields $\len M\leq\len N$. By \Rem{R:Vasc},  the epimorphism $M\onto N$ must then be an isomorphism.
\end{proof}

The following result generalizes Miyata's result \cite{Miy} as we do not need to assume that the given sequence is left exact.

\begin{theorem}\label{T:miyata}
An exact sequence $M\to N\to C\to 0$ is split exact \iff\    $N\iso M\oplus C$.
\end{theorem} 
\begin{proof} One direction is just the definition of split exact. 
Let $\bar M$ be the image of $M$ and apply \Thm{T:semadd} to  $\Exactseq {\bar M}NC$ to get $\len N\leq\len C\ssum\len {\bar M}$. On the other hand, $N\iso M\oplus C$ yields $\len N=\len M\ssum\len C$, whence $\len M\leq\len {\bar M}$. Since $\bar M$ is a homomorphic image of $M$, they must be isomorphic by \Rem{R:Vasc}. Hence, we showed $M\to N$ is injective.   At this point we could invoke  \cite{Miy}, but we can as easily give a direct proof of splitness as follows.  Given a finitely generated $R$-module $H$,  since $M\tensor H\iso (N\tensor H)\oplus (C\tensor H)$,   the same argument  applied to the tensored exact sequence 
$$
M\tensor H\to  N\tensor H\to C\tensor H\to 0,
$$
  gives the injectivity of  the first arrow. We therefore showed that $M\to N$ is pure, whence split by \cite[Theorem 7.14]{Mats}. 
\end{proof}

\begin{theorem}\label{T:nonsingsub}
Let $X$ be a non-singular variety over an \acf\ $k$. Then a closed subscheme $Y\sub X$ with ideal of definition $\mathcal I$ is non-singular \iff\ $\Omega_{X/k}\tensor\loc_Y$ is locally isomorphic to $\mathcal I/\mathcal I^2\oplus \Omega_{Y/k} $.
\end{theorem}
\begin{proof}
Since $X$ is non-singular, its module of differentials $\Omega_{X/k}$ is locally free (\cite[Theorem 8.15]{Hart}), whence so is $\Omega_{X/k}\tensor\loc_Y$, and therefore so is its direct summand $\Omega_{Y/k}$. Moreover, by \Thm{T:miyata}, the conormal sequence
$$
\mathcal I/\mathcal I^2\to \Omega_{X/k}\tensor\loc_Y\to \Omega_{Y/k}\to 0
$$
 is then split exact, and the result now follows from \cite[Theorem 8.17]{Hart}.
\end{proof}  

\begin{theorem}\label{T:retract}
Let $A$ be a finitely generated $R$-algebra,   $I\sub A$ an ideal, and $\bar A:=A/I$.  The closed immersion $\op{Spec}\bar A\sub \op{Spec}(A/I^2)$ is a retract over $R$ \iff\ we have an isomorphism of $A$-modules
\begin{equation}\label{eq:condiff}
\Omega_{A/R}\big/I\Omega_{A/R}\iso \Omega_{\bar A/R}\oplus I/I^2.
\end{equation} 
\end{theorem} 
\begin{proof}
One direction is easy, and if \eqref{eq:condiff} holds, then the conormal sequence
$$
  I/  I^2\to \Omega_{A/R}\big/I\Omega_{A/R}\to \Omega_{\bar A/R}\to 0
$$
is split exact by \Thm{T:miyata}, so that the result follows from \cite[Proposition 16.12]{Eis}.
\end{proof} 

\begin{remark}\label{R:cotan}
We can also formulate a necessary and sufficient condition for the cotangent sequence to be split, although I do not know of any consequences of this fact: given \homo{s} $R\to S\to T$, the sequence 
$$
 {T\tensor_S\Omega_{S/R}}\to{\Omega_{T/R}}\to{\Omega_{S/R}}\to 0
$$
is split exact \iff\ ${\Omega_{T/R}}\iso ({T\tensor_S\Omega_{S/R}})\oplus{\Omega_{S/R}}$.
\end{remark}

%
\begin{proposition}\label{P:complc}
Let $(R,\maxim)$ be a $d$-dimensional local \CM\ ring with canonical module $\omega_R$. Assume there exist exact sequences $\Exactseq NMX$ and $\Exactseq MNY$. For all $e$ such that  $\op{dim}X<e\leq \op{depth}Y$, we have 
$$
\ext R{d-e}M{\omega_R}\iso \ext R{d-e}N{\omega_R}.
$$
\end{proposition}
\begin{proof}
By faithfully flat descent, we may pass to the completion of $R$, and therefore assume from the start that $R$ is complete. By Grothendieck vanishing, the local cohomology groups $\op H^e_\maxim(X)$ and $\op H^{e-1}_\maxim(Y)$ vanish. Taking  local cohomology of the two respective sequences therefore yields
\begin{align*}
&\op H^e_\maxim(N)\to \op H^e_\maxim(M)\to \op 0=H^e_\maxim(X)\\
\op H^{e-1}_\maxim(Y)=0&\to \op H^e_\maxim(M)\to \op H^e_\maxim(N)
\end{align*}
Taking Matlis duals and using Grothendieck duality (\cite[Theorem 3.5.8]{BH}), we get exact sequences
\begin{align*}
0\to&\ext R{d-e}M{\omega_R}\to \ext R{d-e}N{\omega_R}\\
&\ext R{d-e}N{\omega_R}\to \ext R{d-e}M{\omega_R}\to 0
\end{align*}
and the result now follows from \Cor{C:subim}.
\end{proof}

Inspired by the previous results, we introduce the following measures for two modules $M$ and $N$ to be non-isomorphic: let $\kappa(M,N)$ (respectively, $\gamma(M,N)$) be the infimum of all $\len{\op{ker}f}$ (respectively,  $\len{\op{coker}f}$) for $f\in\hom RMN$. We may rephrase \Cor{C:subim} as
%
\begin{equation}\label{eq:noniso}
  M\iso N   \qquad\text{\iff} \qquad\gamma(M,N)+\kappa(M,N)=0. 
\end{equation} 

In fact, as the previous examples suggest, $\kappa$ is often bounded by $\gamma$: 
 
\begin{lemma}\label{L:gk}
If $\lenmod RM=_e\lenmod RN$ and $\deg\gamma(M,N)\leq e-1$, for some $e$, then $\deg\kappa(M,N)\leq e-1$.
\end{lemma} 
\begin{proof}
By assumption, we can find a morphism $M\to N$ whose cokernel $C$ has length $\gamma=\gamma(M,N)$. We use the notation from \eqref{eq:splitord} but drop the subscript $e$ as this will not change throughout the proof.  
Let $K$ and  $\bar M$  the respective kernel and image, and let $\kappa,\bar\mu,\mu,\nu$ be the lengths of $K,\bar M,M,N$.   By assumption, $\mu^+=\nu^+$ and $\gamma=\gamma^-$.  By \Thm{T:semadd}, the exact sequence $\Exactseq {\bar M}NC$ gives 
\begin{equation}\label{eq:abq}
\gamma^-+\nu^++\nu^-\leq \bar\mu^++\bar\mu^-\leq \nu^+\ssum\gamma^-\ssum\nu^-,
\end{equation} 
 from which it follows that $\nu^+=\bar\mu^+$. Applying semi-additivity instead to $\Exactseq KM{\bar M}$ gives $\bar\mu+\kappa\leq\mu$. Since $\bar\mu^+=\nu^+=\mu^+$, we get $\bar\mu=_e\mu $, and hence   $\deg\kappa<e$ by \Lem{L:eineq}. Since $\kappa(M,N)\leq \kappa$, our claim follows.
\end{proof} 

Let us say that $M$ and $N$ are \emph{isomorphic at level $e$}, denoted $M\iso_e N$, if there exists a morphism $M\to N$, called an \emph{isomorphism at level $e$}, whose kernel and cokernel both have  dimension  strictly less than $e$. We similarly define an \emph{epimorphism at level $e$} as one whose cokernel has dimension strictly less than $e$.  Of course $\iso_0$ just means isomorphic,   whereas $\iso_{\op{dim}R}$ gives the notion of being  generically isomorphic.  

\begin{proposition}\label{P:isolevel}
 For $e$ equal to  the degree of $\kappa(M,N)+\gamma(M,N)$, we have an isomorphism $M\iso_{e+1}N$ at level $e+1$.  
\end{proposition} 
\begin{proof}
By assumption, there exist $f\colon M\to N$ and $g\colon M\to N$ such that $K:=\op{ker}(f)$ and $C:=\op{coker}(g)$ have respective  lengths $\kappa(M,N)$ and $\gamma(M,N)$. By \Thm{T:dim}, this means that $K$ and $C$ have dimension at most $e$. Let $\pr$ be an arbitrary prime ideal of dimension strictly bigger than $e$. Since $K_\pr$ and $C_\pr$ are then both zero, $\kappa(M_\pr,N_\pr)=\gamma(M_\pr,N_\pr)=0$, and hence $M_\pr\iso N_\pr$ by \eqref{eq:noniso}. Since they have therefore the same length, the epimorphism $g_\pr\colon M_\pr\to N_\pr$ must be an isomorphism by \Rem{R:Vasc}. Let $H$ be the kernel of $g$. We showed that $H_\pr=0$. Since this holds for all $\pr$ of dimension $>e$, we must have $\op{dim}(H)\leq e$, showing that $g$ is an isomorphism at level $e+1$.  
\end{proof} 

\begin{corollary}\label{C:isolevel}
Given $e\geq 0$, we have an isomorphism $M\iso_{e+1}N$  at level $e+1$  \iff\ $\lenmod RM=_{e+1}\lenmod RN$ and $\gamma(M,N)\leq e$.
\end{corollary} 
\begin{proof}
The non-trivial direction   follows from \Prop{P:isolevel} and \Lem{L:gk}.
\end{proof}

\section{Length criterion for acyclicity}
Given a   complex     
\begin{equation}\label{eq:complex}
 \mathcal M \colon 0\to M_t\to M_{t-1}\to\dots\to M_1\to M_0\to 0
\end{equation} 
let us define its \emph{lower length} $\lowlen{\mathcal M}$ and its \emph{upper length} $\hilen{\mathcal M}$ as the ordinals
\begin{align*}
\lowlen {\mathcal M}&:= \sum_{i\equiv t+1\mod 2} \lenmod R{M_i}\\
\hilen{\mathcal M}&:= \Ssum_{i\equiv t\mod 2} \lenmod R{M_i},
\end{align*}
where we use the ascending order of the index set in the first sum.

\begin{lemma}\label{L:lowhi}
If $\mathcal M$ is exact, then $\lowlen{\mathcal M}\leq \hilen{\mathcal M}$.
\end{lemma} 
\begin{proof}
We can break the   exact sequence $\mathcal M$ into short exact sequences
\begin{equation}\label{eq:shex}
 \Exactseq{Z_i}{M_{i}}{Z_{i-1}}
\end{equation} 
where $Z_0=M_0$ and $Z_{t-1}=M_t$. Let $\zeta_i$ and $\mu_i$ be the respective lengths of $Z_i$ and $M_i$. By \Thm{T:semadd}, we have for each $i$ an inequality
\begin{equation}\label{eq:zetai1}
\zeta_{i-1}+\zeta_i\leq\mu_i\leq\zeta_{i-1}\ssum\zeta_i.
\end{equation} 
If $t$ is even, then $\lowlen{\mathcal M}=\mu_1+\mu_3+\dots+\mu_{t-1}$. Using  respectively the upperbounds in \eqref{eq:zetai1} for odd $i$,  ordinal arithmetic, and the lowerbounds in \eqref{eq:zetai1} for even $i$,  gives 
\begin{align*}
\lowlen{\mathcal M}&\leq (\zeta_0\ssum\zeta_1)+(\zeta_2\ssum\zeta_3)+\dots+(\zeta_{t-2}\ssum\zeta_{t-1})\\
&\leq \mu_0\ssum (\zeta_1+\zeta_2)\ssum  \dots\ssum(\zeta_{t-3}+\zeta_{t-2})\ssum\mu_t\\
&\leq \mu_0\ssum \mu_2\ssum  \dots\ssum\mu_{t-2}\ssum\mu_t=\hilen{\mathcal M}.
\end{align*}
The proof for $t$ odd is similar and left to the reader. 
\end{proof}


\begin{theorem}\label{T:acycl}
Let   $e\geq -1$  and suppose that all homology groups $\op H_i(\mathcal M)$ for $i<t$ have dimension at most $e$, where $\mathcal M$ is the complex~\eqref{eq:complex}.
 If $\hilen{\mathcal M}\aleq e \lowlen{\mathcal M}$,
   then $\op H_t(\mathcal M)$ too has dimension at most $e$.
\end{theorem} 
\begin{proof}
 The homology $H_i:=\op H_i(\mathcal M)$ of $\mathcal M$, is given, for each $i<t$, by two short exact sequences
\begin{align}
\label{i:hi}&\Exactseq{B_i}{Z_{i}}{H_i}\\
\label{i:bi}&\Exactseq{Z_{i+1}}{M_{i+1}}{B_{i}} 
\end{align} 
with $Z_{0}=M_0$ and $Z_{t}=H_t$. We use the notation~\eqref{eq:splitord} with the value $e+1$, but dropping  the subscript. Let $\mu_i$, $\theta_i$, $\beta_i$, and $\zeta_i$ be the respective lengths of $M_i$, $H_i$, $B_i$, and $Z_i$. By   \Thm{T:dim}, we have $\theta_i^+=0$, for all $i<t$, and we   want to show the same for $i=t$.
   By semi-additivity, \eqref{i:hi} yields
\begin{equation}\label{eq:thetai}
\theta_i+\beta_i\leq \zeta_{i}\leq  \theta_i\ssum \beta_i.
\end{equation} 
Since $\deg\theta_i\leq e$, for all $i<t$, we get $\beta_i^+=\zeta_{i}^+$. 
  Semi-additivity applied to \eqref{i:bi} for $i-1$ gives
$$
\beta_{i-1}+\zeta_{i}\leq \mu_i\leq  \beta_{i-1}\ssum\zeta_{i},
$$
and hence, for $i<t$, we have 
\begin{equation}\label{eq:betai}
\zeta_{i-1}^++\zeta_i^+\leq \mu_i^+\leq \zeta_{i-1}^+\ssum\zeta_i^+,
\end{equation} 
and for $i=t$, using that $Z_{t}=H_t$, we get
\begin{equation}\label{eq:betat}
\zeta_{t-1}^++\theta_t^+\leq \mu_t^+.
\end{equation} 
 For simplicity, let us assume $t$ is even (the odd case is similar). Using respectively the upperbounds in \eqref{eq:betai} for odd $i<t$, ordinal arithmetic,   the lowerbounds in \eqref{eq:betai} for even $i<t$, and then for $i=t$, we get   inequalities
\begin{equation}\label{eq:lohi}
\begin{aligned}
 \lowlen{\mathcal M}^+&=\mu_1^++\mu_3^++\dots+\mu_{t-1}^+\\
 &\leq (\zeta_0^+\ssum\zeta_1^+)+(\zeta_2^+\ssum\zeta_3^+)+\dots+(\zeta_{t-2}^+\ssum\zeta_{t-1}^+)\\
 &\leq \zeta_0^+\ssum(\zeta_1^++\zeta_2^+)\ssum\dots\ssum (\zeta_{t-3}^++\zeta_{t-2}^+)\ssum \zeta_{t-1}^+\\
 &\leq \mu_0^+\ssum \mu_2^+\ssum\dots\ssum\mu_{t-2}^+\ssum \zeta_{t-1}^+\\
 &\leq   \mu_0^+\ssum \mu_2^+\ssum\dots\ssum\mu_{t-2}^+\ssum \mu_t^+=\hilen{\mathcal M}^+
\end{aligned}
\end{equation} 
(note that $\beta_0^+=\zeta_{0}^+=\mu_0^+$).  By assumption, the lowerbound is bigger than or equal to the upperbound, so that we have equalities throughout. In particular, the last of these gives $\zeta_{t-1}^+=\mu_t^+$.  Applied to \eqref{eq:betat}, we then get the desired $\theta_t^+=0$.
%
\end{proof} 

Some special cases of this result are worth mentioning separately:
 
\begin{corollary}\label{C:acycl2}
Let 
$\mathcal S\colon \Exactseq {M_2}{M_1}{M_0}$ be    a complex
     such that $\len {M_0}\ssum\len {M_2}\leq \len {M_1}$. If   $\mathcal S$ is  right exact (respectively, right exact at all but finitely many maximal  ideals), then  $\mathcal S$ is    exact (respectively,   exact at all but finitely many maximal  ideals).\qed
\end{corollary} 

\begin{corollary}\label{C:pdim1}
Let $(R,\maxim)$ be   local ring with residue field $k$ and let $d$ and $\rho$ be its respective dimension and length. Let  $b_i:=\dim_k(\tor RiMk)$ be the   Betti numbers of $M$. If $\lenmod RM=(b_0-b_1)\odot\rho$, then $M$ has projective dimension one, and the converse is true if $M$ is moreover unmixed of dimension $d$. If $M$ is unmixed of dimension $d$ and $\lenmod RM=(b_0-b_1+b_2)\odot\rho$, then    $M$ has projective dimension two.
\end{corollary} 
\begin{proof}
By assumption, there exists an exact sequence $\mathcal S_1\colon R^{b_1}\to R^{b_0}\to M\to 0$ (respectively,   $\mathcal S_2\colon R^{b_2}\to R^{b_1}\to R^{b_0}\to M\to 0$).
  If   $\mu=\len M$, then $\lowlen{\mathcal S_1}=b_0\odot\rho$ and $\hilen{\mathcal S_1}=\mu\ssum (b_1\odot\rho)$, proving that $M$ has projective dimension one, since $\mathcal S_1$ is then left exact by Corollary~\ref{C:acycl2}. 
 
Suppose next that  $M$ is   unmixed of dimension $d$, so that $\mu=q\omega^d$ by \Thm{T:unm}. Hence if $\mathcal S_1$ is also exact  on the left, then $q\omega^d+(b_1\odot\rho)\leq b_0\odot\rho \leq q\omega^d\ssum (b_1\odot\rho)$ by \Thm{T:semadd}, and both bounds are equal by ordinal arithmetic.

To prove the second case, assume $q\omega^d=\mu=(b_0+b_2-b_1)\odot\rho$, from which it follows that $\mu+(b_1\odot\rho)=(b_0+b_2)\odot\rho$. Since the latter two ordinals are respectively $\lowlen{\mathcal S_2}$ and $\hilen{\mathcal S_2}$, the sequence  $\mathcal S_2$ is also exact on the left by \Thm{T:acycl}.  
\end{proof} 


\begin{corollary}\label{C:period}
Let $M$ and $N$ be  modules of the same dimension with $N$ moreover unmixed. If  the complex 
 $\mathcal S\colon0\to N\to M\to M\to N\to 0$
is exact at all spots except possibly at the left most one, then it is in fact exact.
\end{corollary} 
\begin{proof}
Let $\mu$ and $\nu$ be the respective lengths of $M$ and $N$. By \Thm{T:unm}, we get   $\lowlen{\mathcal S}=\nu+\mu=\nu\ssum\mu=\hilen{\mathcal S}$. The result now follows from \Thm{T:acycl}.
\end{proof} 

Given a complex \eqref{eq:complex}, let us define its \emph{generic Euler \ch} $\genchi M$ as the alternating sum $\sum_i (-1)^i\genlen{}{M_i}$.

\begin{corollary}\label{C:acycunm}
If $\mathcal M\colon    M_t\to M_{t-1}\to\dots\to M_1\to M_0\to 0$   is an exact sequence in which all modules are unmixed of dimension $d$, then the generic length of $\op H_t(\mathcal M)$ is equal to $\genchi M$. In particular, $M_t\to M_{t-1}$ is injective \iff\ $\genchi M=0$.
\end{corollary} 
\begin{proof}
Since $\op H_t(\mathcal M)$ is the kernel of $M_t\to M_{t-1}$, it suffices to show the last assertion. By \Thm{T:unm},   ordinal sum and shuffle sum are the same here, so that $\genchi M=0$ \iff\ $\lowlen {\mathcal M}=\hilen{\mathcal M}$, and the result now follows from \Thm{T:acycl}.
\end{proof}

Recall that when $(R,\maxim)$ is local, any finitely generated $R$-module $M$ has a uniquely defined \emph{syzygy}, denoted $\Omega M$, given by a   \emph{minimal exact sequence}  $\Exactseq{\Omega M}{R^n}M$, that is to say,  such that $\Omega M\sub\maxim R^n$.

\begin{theorem}\label{T:syz}
Let   $\Exactseq{\Omega M}XM$ be an exact sequence. If $M$ is unmixed of maximal dimension, then $X$ is free \iff\ $\Omega M\sub\maxim X$.
\end{theorem} 
\begin{proof}
If $X$ is free, then the result follows from the minimality and uniqueness of   syzygies. Let $N:=\Omega M$, and let $\mu$, $\nu$, and $\chi$ be the respective lengths of $M$, $N$, and $X$.   By applying \Cor{C:unm} to the given exact sequence and to the minimal exact sequence $\Exactseq N{R^n}M$ respectively, we get  $\chi=\mu+\nu=\len{R^n}$. Since $N\sub \maxim X$, it follows from Nakayama's lemma, that $X$ and $M$ have the same minimal number of generators, which by minimality is precisely $n$. Hence there exists a surjective morphism $R^n\onto X$. As both have the same length, this must be an isomorphism by Remark~\ref{R:Vasc}.
\end{proof}

\begin{theorem}\label{T:pullback}
   Let $\varphi\colon \tilde M\to M$ be a surjective morphism, $N\sub M$ a submodule, and $\tilde N=\inverse\varphi N$ its pull-back inside $\tilde M$. Let $\psi\colon\tilde N\to \tilde M$ be an arbitrary morphism with cokernel $C$. If  $M$ is  unmixed of maximal dimension, and there exists an exact sequence $\Exactseq CMN$, then $\psi$ is injective.
\end{theorem}
\begin{proof}
By assumption, we have an exact sequence
 $$
 \mathcal S\colon \tilde N\map{\psi} \tilde M\to M\to N\to 0.
 $$
 In particular,  $\lowlen {\mathcal S}=\len N+\len{\tilde M} $ and $\hilen{\mathcal S}=\len M\ssum\len{\tilde N}$.  Let $K$ be the kernel of $\varphi$, which is then also the kernel of the restriction of $\varphi$ to $\tilde N$. In other words, we have exact sequences $\Exactseq K{\tilde M}M$ and $\Exactseq K{\tilde N}N$. Since $M$ is unmixed, so is $N$, and hence   \Cor{C:unm}   yields $\len{\tilde M}=\len M\ssum\len K$ and $\len{\tilde N}=\len N\ssum\len K$. Moreover, by unmixedness  $\lowlen{\mathcal S}$ is equal to $\len N\ssum\len{\tilde M} $ whence equal to  $\hilen{\mathcal S}$. Therefore,  $\psi$ is injective by \Thm{T:acycl}.
\end{proof}

%

\section{Locally isomorphic modules}
Given two modules $M$ and $N$, define $\mathfrak k(M,N)$ as the sum of all $\ann R{\op{ker}(f)}$, for $f\in\hom RMN$, and similarly, define    $\mathfrak c(M,N)$ as the sum of  all $\ann R{\op{coker}(f)}$. In particular, if $N$ is a homomorphic image of $M$, then $\mathfrak c(M,N)=1$, and if $M$ is (isomorphic to) a submodule of $N$, then $\mathfrak k(M,N)=0$.


\begin{lemma}\label{L:cancunit}
 $M$ and $N$ are locally isomorphic \iff\ $\mathfrak c(M,N) \mathfrak c(N,M)=1$.
\end{lemma} 
\begin{proof}
For each prime ideal $\pr$, we can find surjective morphisms $M_\pr\to N_\pr$ and $N_\pr\to M_\pr$. The composition must then be an isomorphism by \Cor{C:Vasc}. The converse is also immediate.
\end{proof}

Given an $R$-module $M$, let $\assid RM$ be the sum of all its associated primes.  

 \begin{lemma}\label{L:submod}
If $\mathfrak c(M,N)$ is not contained in $\assid RN$, then $\lenmod RN\leq \lenmod RM$.
If $R$, moreover, is a complete local ring and $\mathfrak k(M,N)$ is not contained in $\assid RM$, then $\lenmod RM\leq \lenmod RN$.
\end{lemma} 
\begin{proof}
Write $\mathfrak c(M,N)$ as a finite sum $\sum_i \ann R{\op{coker}(f_i)}$, for some   $f_i\colon M\to N$. By assumption this is not contained in the sum of all associated primes of $N$. Therefore, there must be some $i$,  such that $\ann R{\op{coker}(f_i)}$ contains an $N$-regular element $x$. Let $f:=f_i$. Since then  $xN\sub f(M)$, and since $x$ is $N$-regular, the morphism $N\to f(M)$ sending $a\in N$ to $xa$ is an injection. In particular, $\len N\leq \len{f(M)}$ by semi-additivity. Since $f(M)$ is a  homomorphic image of $M$, its length is at most $\len M$.

To prove the last assertion, under the additional assumption that $R$ is complete and local, with residue field $k$, we will use Matlis duality, where we write $\matlis M:=\hom RME$ for the Matlis dual  of a module $M$, with $E$ the injective hull of $k$.  By the same argument, there exists $g\colon M\to N$ and an $M$-regular element $x$ such that $xK=0$, where $K$ is the kernel of $g$. Taking Matlis duals, we get an exact sequence
$$
\matlis N\map{\matlis g}\matlis M\to \matlis K\to 0.
$$
Since $x\matlis K=0$, we get $x\matlis M\sub W:=\matlis g(\matlis N)$. On the other hand, since $0\to M\map x M$ is injective, the dual map $\matlis M\map x\matlis M\to 0$ is surjective, so that $x\matlis M=\matlis M$. Taking again duals then yields an epimorphism  $\matlis W\onto M$, so that $\len M\leq \len{\matlis W}$. Since we also have an epimorphism $\matlis N\onto W$, we get an embedding $\matlis W\sub N$, so that $\len{\matlis W}\leq\len N$.
%
%
\end{proof} 

\begin{theorem}\label{T:lociso}
Let  $M$ and $N$ be finitely generated $R$-modules. If  
\begin{enumerate}
\item\label{i:MN}  $\mathfrak c(M,N)$ is not contained in $\assid RN$;
\item\label{i:NM} $\mathfrak c(N,M)$ is not contained in $\assid RM$;
\item\label{i:1} $\mathfrak c(M,N)+\mathfrak c(N,M)=1$,
\end{enumerate}
 then $M$ and $N$ are locally isomorphic. 
\end{theorem} 
\begin{proof}
Assume first that \eqref{i:MN}--\eqref{i:1} hold, and let $\pr\sub R$ be an arbitrary prime ideal. From \eqref{i:1}, it follows that $\pr$ does not contain some annihilator of a cokernel of a morphism (in either direction) between $M$ and $N$. Since the conditions are symmetric in $M$ and $N$, we may assume without loss of generality that we have a morphism  $f\colon M\to N$ such that $\ann R{\op{coker}(f)}$ does not contain $\pr$. This means that $f_\pr\colon M_\pr\to N_\pr$ is surjective. It is not hard to see that $\mathfrak c(N_\pr,M_\pr)$ contains $\mathfrak c(N,M)$ and $\assid{R_\pr}{M_\pr}$ is contained in $\assid RMR_\pr$. Hence by \eqref{i:NM}, we see that $\mathfrak c(N_\pr,M_\pr)$ is not contained in $\assid{R_\pr}{M_\pr}$, and so 
\begin{equation}\label{eq:MNpr}
\lenmod{R_\pr}{M_\pr}\leq\lenmod {R_\pr}{N_\pr}
\end{equation} 
 by \Lem{L:submod}. It now follows from \Rem{R:Vasc} that $f_\pr$ is an isomorphism.
%
\end{proof} 

\begin{remark}\label{R:lociso}
We may replace in \eqref{i:MN} and \eqref{i:NM} the ideals $\mathfrak c$ by the ideals $\mathfrak d$ defined on a pair of modules as the sum $\mathfrak d(M,N):=\mathfrak c(M,N)+\mathfrak k(N,M)$. Indeed, since the problem is local, we may localize $R$ so that $\pr$ is its maximal ideal.   If $\mathfrak k(M,N)$ is not contained in $\assid {}M$, then the second part of  \Lem{L:submod}  gives \eqref{eq:MNpr}, at least over the completion $\complet R$ of $R$. Hence $\complet f$ is an isomorphism, whence so is $f$   by faithfully flat descent.
\end{remark} 

\begin{corollary}\label{C:canc}
Let $M$ and $N$ be such that $\mathfrak c(M,N)+\mathfrak c(N,M)=1$. If there exists some $H$ such that $M\oplus H$ and $N\oplus H$ are locally isomorphic, then $M$ and $N$ are already locally isomorphic.
\end{corollary}
\begin{proof}
By assumption, any prime ideal $\pr$ does not contain either $\mathfrak c(M,N)$ or $\mathfrak c(N,M)$. Let us say the former holds. Localizing at $\pr$, we may assume $R$ is local and $N$ is a homomorphic image of $M$. Taking lengths, we get $\len M\ssum\len H=\len N\ssum \len H$ by semi-additivity, and hence $\len M=\len N$, whence $M\iso N$ by \Rem{R:Vasc}.
\end{proof}

\section{Appendix: shuffle sums}\label{s:ssum}

 Recall that neither addition nor multiplication  of ordinals is commutative. We  will give three different but equivalent ways of defining a different, commutative addition
operation on $\ord$, which we temporarily will denote as $\ssum$, $\bsum$ and
$\tsum$. The sum $\ssum$ is also known as the \emph{natural (Hessenberg)
sum} and is often denoted $\#$. Recall our convention of writing multiplication from left-to-right (see \S\ref{s:Ord}). Every ordinal $\alpha$ can be written as a sum
\begin{equation}\label{eq:CNF}
\alpha=a_n\omega^{\nu_n}+\dots+a_1\omega^{\nu_1}
\end{equation}
where the $\nu_i$ (called the \emph{exponents}) form a strictly ascending
chain of ordinals, that is to say, $\nu_1<\dots<\nu_n$, and the $a_i$ (called the
\emph{coefficients})
are non-negative integers. This decomposition (in base $\omega$) is unique if 
we moreover require that all coefficients $a_i$ are non-zero, called
the \emph{Cantor normal form} (\emph{in base $\omega$}) of $\alpha$. 
If \eqref{eq:CNF} is in Cantor normal form, then we call the highest (respectively, lowest) occurring exponent, the \emph{degree} (respectively,   the \emph{order}) of $\alpha$ and we denote these respectively by $\op{deg}(\alpha):=\nu_n$ and  $\low\alpha:=\nu_1$.
Note that $\alpha$ is a successor ordinal \iff\
$\low\alpha=0$.  

Given a second
ordinal $\beta$, we may assume that
after possibly adding some more exponents, that it can also be written in the
form~\eqref{eq:CNF}, with coefficients $b_i\geq 0$ instead of the $a_i$. We now
define 
$$
\alpha\ssum\beta:=
(a_n+b_n)\omega^{\nu_n}+\dots+(a_1+b_1)\omega^{\nu_1}.
$$
It follows that $\alpha\ssum\beta$ is equal to $\beta\ssum\alpha$ and is
greater than or equal to both $\alpha+\beta$ and   $\beta+\alpha$. For
instance if $\alpha=\omega+1$ then $\alpha\ssum\alpha=2\omega+2$
whereas $\alpha+\alpha=2\omega+1$. In case both
ordinals are finite, $\alpha\ssum\beta=\alpha+\beta$. It is easy to check
that we have the following \emph{finite distributivity property}:
\begin{equation}\label{eq:findist}
(\alpha\ssum\beta)+1=(\alpha+1)\ssum\beta=\alpha\ssum(\beta+1).
\end{equation}
In fact, this follows from the more general property that
$(\alpha\ssum\beta)+\theta=(\alpha+\theta)\ssum\beta=
\alpha\ssum(\beta+\theta)$ for all $\theta<
\omega^{o+1}$, where $o$ is the minimum of $\low\alpha$ and
$\low\beta$.

For the second definition, we use transfinite induction on the pairs 
$(\alpha,\beta)$ ordered
lexicographically, that is to say, induction on the ordinal $\alpha\beta$.
Define $\alpha \bsum0:=\alpha$ and $0\bsum\beta:=\beta$ so that
we may assume $\alpha,\beta>0$.  If $\alpha$ is a
successor ordinal (recall that its predecessor is then denoted $\alpha-1$), then we define $\alpha\bsum\beta$ as
$((\alpha-1)\bsum\beta)+1$. Similarly, if $\beta$ is a successor
ordinal, then we define $\alpha\bsum\beta$ as $(\alpha\bsum(\beta-1))+1$.
Note that by transfinite induction,  both definitions agree when both $\alpha$ and $\beta$ are successor
ordinals, so that we have no ambiguity in defining this sum operation when at
least one of the components is a successor ordinal. So remains the case that
both are limit ordinals. If  $\low\alpha\leq\low\beta$, then we  let $\alpha \bsum \beta$
be equal to the supremum of the $\delta \bsum \beta$ for all
$\delta<\alpha$. In the remaining case, when  $\low\alpha>\low\beta$,   we  let $\alpha \bsum \beta$
be equal to the supremum of the $\alpha \bsum \delta$ for all
$\delta<\beta$. This concludes the definition of $\bsum$.
 
 Finally, define
$\alpha \tsum \beta$ as
the supremum of all sums $\alpha_1+\beta_1+\dots+\alpha_n+\beta_n$, where the
supremum is taken  over all $n$ and all decompositions
$\alpha=\alpha_1+\dots+\alpha_n$ and $\beta=\beta_1+\dots+\beta_n$. Loosely speaking,
$\alpha \tsum \beta$ is the
largest possible ordering one can obtain by   \emph{shuffling} pieces of 
$\alpha$ and $\beta$. Since we may take
$\alpha_1=0=\beta_n$, one checks that
$\alpha \tsum \beta=\beta \tsum \alpha$.

\begin{theorem}\label{T:ssum}
For all ordinals $\alpha,\beta$ we have
$\alpha\ssum\beta=\alpha\bsum\beta=\alpha\tsum\beta$.
\end{theorem}
\begin{proof}
Let $\gamma:=\alpha\ssum\beta$, $\bar\gamma:=\alpha \bsum \beta$ and
$\tilde\gamma:=\alpha \tsum \beta$. 
We first prove $\gamma=\bar\gamma$ by induction on  
$\alpha\beta$. Since the case $\alpha=0$ or $\beta=0$ is trivial,  we may take
$\alpha,\beta>0$. If   $\alpha$ is a successor ordinal, then
\begin{equation*}
\bar\gamma=((\alpha-1) \bsum \beta)+1=((\alpha-1)\ssum\beta)+1=\alpha\ssum\beta=\gamma,
\end{equation*} 
where the first equality is by definition, the second by
induction and the third by  the finite distributivity
property~\eqref{eq:findist}.  Replacing the role of $\alpha$ and $\beta$, the 
same argument can be used to treat the case when $\beta$ is a successor
ordinal. So we may assume that both are limit ordinals. There are again two 
cases, namely $\low\alpha\leq\low\beta$ and $\low\alpha>\low\beta$. By
symmetry, the
argument for the second case is similar as for the first, so we will only give
the details for the first case. Write $\alpha$ as $\alpha'+\omega^o$ where
$o:=\low\alpha$. 
By definition,
$\bar\gamma$ is the supremum of all $\delta \bsum \beta$ with
$\delta<\alpha$. A cofinal subset of such $\delta$
are the ones of the form $\alpha'+\theta$  
with $0<\theta<\omega^o$, so that $\bar\gamma$ is the
supremum of all $(\alpha'+\theta) \bsum \beta$ for $0<\theta<\omega^o$.
By induction, $\bar\gamma$ is the supremum of all 
\begin{equation}\label{eq:g'}
(\alpha'+\theta)\ssum\beta=(\alpha'\ssum\beta) +\theta,
\end{equation} 
where the equality holds because $o\leq \low\beta$.
Taking the supremum of the ordinals in \eqref{eq:g'} for $\theta<\omega^o$, we
get that $\bar\gamma=(\alpha'\ssum\beta)+\omega^o$. Using the remark following
\eqref{eq:findist} one checks that this is
just $(\alpha'+\omega^o)\ssum\beta=\alpha\ssum\beta=\gamma$. 

The inequality $\gamma\leq\tilde\gamma$ is clear using the shuffle of the
terms in the Cantor normal forms~\eqref{eq:CNF} for $\alpha$ and $\beta$.
To finish the proof,
we therefore need to show, by induction on $\alpha$, that
\begin{equation}\label{eq:g3}
\alpha_1+\beta_1+\dots+\alpha_n+\beta_n\leq\bar\gamma,
\end{equation}
for all decompositions $\alpha=\alpha_1+\dots+\alpha_n$ and
$\beta=\beta_1+\dots+\beta_n$. Since $\tsum$ is commutative, we may assume
$\low\alpha\leq\low\beta$ and, moreover,   that $\alpha_n>0$.  Suppose
first that $\alpha$ is a successor ordinal. In particular,
$\alpha_n$ is also a successor ordinal. By
definition, $\bar\gamma=((\alpha-1) \bsum \beta)+1$. Using
the decomposition $\alpha-1=\alpha_1+\dots+\alpha_{n-1}+(\alpha_n-1)$ and
induction, we get that 
$\alpha_1+\beta_1+\dots+\beta_n+(\alpha_n-1)\leq (\alpha-1) \bsum \beta$. Taking
successors of both ordinals then yields \eqref{eq:g3}. Hence suppose $\alpha$
is a limit ordinal. Let $\theta<\alpha_n$ and apply the induction to each 
$\delta:=\alpha_1+\dots+\alpha_{n-1}+\theta$, to get
$$
\alpha_1+\beta_1+\dots+\beta_{n-1}+\theta+\beta_n\leq \delta \bsum  \beta.
$$
Taking suprema of both sides then yields inequality  \eqref{eq:g3}. 
\end{proof}

We will denote this new sum simply by $\ssum$ and refer to
it as the \emph{shuffle sum} of two ordinals, in view of its third equivalent
form. 
 
\providecommand{\bysame}{\leavevmode\hbox to3em{\hrulefill}\thinspace}
\providecommand{\MR}{\relax\ifhmode\unskip\space\fi MR }
\providecommand{\MRhref}[2]{%
  \href{http://www.ams.org/mathscinet-getitem?mr=#1}{#2}
}
\providecommand{\href}[2]{#2}

%
%

\end{document}